\documentclass[twoside]{irmaems}
\usepackage{amssymb} %use \usepackage[psamsfonts]{amssymb} if your fonts are from Y&Y/Blue Sky Research
\usepackage{amsmath} %use \usepackage{amsmath,cmmib57} if your fonts are from Y&Y/Blue Sky Research
\usepackage{latexsym}
\usepackage{graphicx}

\renewcommand\theequation{\arabic{section}.\arabic{equation}}

\theoremstyle{definition} %%% for statements in roman typeface
 \newtheorem{definition}{Definition}[section]

\theoremstyle{plain}      %%% for statements in italic typeface
 \newtheorem{proposition}[definition]{Proposition}
 \newtheorem{theorem}[definition]{Theorem}

\def\sech{\mbox{sech}}
\newcommand{\Real}{\mathbb R}

\newcommand{\Hyp}{\mathbb{H}}
\newcommand{\Sph}{\mathbb{S}}
\newcommand{\Z}{\mathbb{Z}}
\newcommand{\Cplex}{\mathbb{C}}
\newcommand{\Vol}{\mbox{Vol}}
\newcommand{\Length}{\mbox{Length}}

\begin{document}
\title{Identities on Hyperbolic Manifolds}
\author{Martin Bridgeman\thanks{This work was partially supported by a grant from the Simons Foundation (\#266344 to Martin Bridgeman)} and  Ser Peow Tan\thanks{Tan was partially supported by the National University
of Singapore academic research grant R-146-000-156-112.}}
\address{
Boston College\\
Chestnut Hill, Ma 02116, USA\\
email:\,\tt{bridgem@bc.edu}
\\[4pt]
National University of Singapore\\
 10 Lower Kent Ridge Road,  S(119076),
Singapore\\
email:\,\tt{mattansp@nus.edu.sg}
}

\maketitle

\begin{abstract}
In this survey, we discuss four classes of identities due principally to Basmajian, McShane, Bridgeman-Kahn and Luo-Tan on hyperbolic manifolds and provide a unified approach for proving them. We also elucidate on the connections between the various identities.
\end{abstract}
%\tableofcontents
\section{Introduction}\label{s-1}
In the last couple of decades, several authors have discovered various remarkable and elegant identities on hyperbolic manifolds, including Basmajian \cite{Bas93}, McShane \cite{McS91, McS98}, Bridgeman-Kahn \cite{B11,BK10} and Luo-Tan \cite{LT11}. Some of these identities have been  generalized and extended, with  different and independent proofs  given in some cases.       In addition to their intrinsic beauty and curiosity value, some of the identities have found important
       applications, in particular, the McShane identity as generalized by Mirzakhani to surfaces with boundary
        played a crucial role in Mirzakhani's computation of the Weil-Petersson volumes of the moduli spaces of
         bordered Riemann surfaces, as well as some subsequent  applications; the Bridgeman-Kahn identity gave lower bounds for the volumes of hyperbolic manifolds with totally geodesic boundary.
Many of these identities were proven independently of each other, for example, although the Basmajian and McShane
identities appeared at about the same time, the authors seemed unaware of each other's work at that point; the
Bridgeman-Kahn identity, although it appeared later was also proven somewhat independently of both of these works.
The exception is the Luo-Tan identity which borrowed inspiration from 
the previous works.

The aim of this article is to explore these identities, to
 analyse the various ingredients which make them work,
and to provide a common framework for which to understand them. We
hope that this will not only put the identities in a more natural
setting and make them easier to understand, but also  point the way
towards a more unified theory with which to view these beautiful
identities, and also point the direction towards possible applications for these identities. 
%In the process, this also provides a more detailed
%explanation and motivation, then that given in \cite{}, of how the
%Luo-Tan identities were derived.

The unifying idea behind all of these identities is fairly simple.
One considers a set $X$ with finite measure $\mu(X)$ associated with
the hyperbolic manifold $M$,  (for example $X=\partial M$, the
boundary of $M$, or $X=T_1(M)$, the unit tangent bundle of $M$) and
look for interesting geometric/dynamical/measure theoretic
decompositions of $X$. Typically, by exploring some geometric or
dynamical aspect of $M$, one can show that the set $X$ decomposes
into a countable union of disjoint subsets $X_i$ of finite non-zero
measure, and a set $Z$ which is geometrically and dynamically
complicated and interesting, but which has measure zero (this is
what we mean by a measure theoretic decomposition of $X$). One
deduces the fact that $Z$ has measure zero from a deep but
well-known result from hyperbolic geometry or dynamical systems. The identity is then just
the tautological equation
$$\mu(X)=\sum_i\mu(X_i).$$
The second part of the problem consists of analysing the sets $X_i$,
in particular, to compute their measures in terms of various
geometric quantities like spectral data. This can be relatively
simple, for example in the case of the Basmajian identity, somewhat
more complicated, like the McShane  and Bridgeman-Kahn identities or
considerably more involved, as in the Luo-Tan identities. A typical
feature is that the sets $X_i$ are indexed by either simple
geometric objects on $M$ like orthogeodesics, or simple subsurfaces
of $M$, like embedded one-holed tori or thrice-punctured spheres. In
particular, their measures depend only on the local geometry and
data, and not on the global geometry of $M$.

One can develop this viewpoint further, for example letting $X$ be the
set of geodesics on $M$ with the Liouville measure  and associating to $X$ the length of the geodesic as a random variable and computing the moment
generating function of this. In this way, for example, the Basmajian
and the Bridgeman-Kahn identities can be viewed as different moments
of the same generating function, see \cite{BT13}.

Alternatively,  as in the case of Bowditch's proof \cite{Bow96} of
McShane's original identity, one can adopt a different viewpoint,
and prove it using a combination of algebraic and combinatorial
techniques. This  has been developed further in \cite{Bow97, Bow98, TWZ06b, TWZ08, TWZ08c} etc, and
provides an interesting direction for further exploration.

Our main aim in this survey is to demystify these identities and to
show that the basic ideas involved in deriving them are very simple.
As such, the exposition will be somewhat leisurely, and where
necessary, we will present slightly different proofs and
perspectives than the original papers. We will refer the reader to
the original papers for the more technical details of computing the
measures $\mu(X_i)$.

The rest of the paper is organized as follows.
In the next couple of sections we first state the four sets of
identities, and then sketch the proofs for these identities from our
perspective. Subsequently, we give a short discussion of the moment
point of view adopted in \cite{BT13} which allows one to view the
Basmajian and Bridgeman identities as different moments of the same
variable and follow this with a short discussion of the Bowditch proof of the McShane identity and subsequent  developments. We conclude the survey
with some open questions and directions for further investigations.

\subsection{Literature}
The literature on the subject is fairly large and growing. To aid the reader we will now give a brief synopsis by identity.

\medskip
\noindent {\bf McShane Identity:} The McShane identity first appeared in McShane's 1991 thesis  \cite{McS91}, ``A remarkable identity for lengths of curves''. This was subsequently generalized (to higher genus surfaces) and published in \cite{McS98}. In the papers \cite{Bow96, Bow97, Bow98}, Bowditch gives a proof of the McShane identity using Markov triples, with extensions to punctured torus bundles and type-preserving quasi-fuchsian representations, see also  \cite{AMS2004, AMS2006} by Akiyoshi, Miyachi and Sakuma for variations. The identity was extended to surfaces with cone singularities in Zhang's  2004 thesis \cite{Zhang04}, see also  \cite{TWZ06}. A Weierstrass points version of the identity was derived by McShane in \cite{McS04}, and the identity was also generalized to closed surfaces of genus two  in \cite{McS06}, using similar techniques. Both of these can also be derived using the  hyperelliptic involution on the punctured torus and  on a genus two surface,  and using the identity on the resulting cone surfaces, as explained by Tan, Wong and Zhang in  \cite{TWZ06}. Mirzakhani gave a proof of the general McShane identity for bordered surface in her 2005 thesis \cite{Mir05} which was subsequently published in  \cite{Mir07}.  In \cite{TWZ06b}, the identity was generalized to the $\mbox{SL}(2,\Cplex)$ case by Tan, Wong and Zhang and for non-orientable surfaces by Norbury in \cite{Nor}. The identity was generalized to $\mbox{PSL}(n,\Real)$ for Hitchin representations  by  Labourie and McShane in \cite{LaMcS09}. A version for two-bridge links was given by Lee and Sakuma in \cite{LS13}. Recent work of Hu, Tan and Zhang in \cite{HTZ13, HTZ13b} have also given new variations and extensions of the identity to the context of Coxeter  group actions on $\Cplex^n$.

\medskip
\noindent{\bf Basmajian Identity:} The Basmajian identity appears in the 1993 paper \cite{Bas93}, ``The orthogonal spectrum of a hyperbolic manifold''. A recent paper of Vlamis \cite{V13} analyses the statistics of the Basmajian identity and derives a formula for the moments of its associated hitting random variable. In the recent paper \cite{PP13}, Paulin and Parkkonen derive formulae for the asymptotic distribution of orthogonal spectrum in a general negatively curved space.
  
\medskip
\noindent{\bf Bridgeman-Kahn Identity:} The Bridgeman-Kahn identity was first proven in the surface case by the first author in the 2011 paper \cite{B11} ``Orthospectra and Dilogarithm Identities on Moduli Space''. An alternate proof was given by Calegari in \cite{Cal10a}. The general case was proven by Bridgeman-Kahn in the paper \cite{BK10}. The paper \cite{Cal10b} of Calegari analyses the connections between  the Bridgeman-Kahn identity and Basmajian identity and gives an orthospectrum identity that has the same form as the Bridgeman-Kahn identity but arises out of a different decomposition. A recent paper of Masai and McShane \cite{MasMcS} has shown that the identity obtained by Calegari is in fact the original Bridgeman-Kahn identity. In the paper  \cite{BT13}  the authors consider the statistics of the Bridgeman-Kahn identity and derive a formula for the moments of its associated hitting random variable. We show that the Basmajian and Bridgeman-Kahn identities arise as the first two moments of this random variable.

\medskip
\noindent{\bf Luo-Tan identity:} The Luo-Tan identity appears in the 2011 preprint \cite{LT11} ``A dilogarithm identity on Moduli spaces of curves''. A version of the identity for small hyperbolic surfaces can be found in \cite{HT13} and  for surfaces with boundary and non-orientable surfaces in \cite{LT13}.

\bigskip

\noindent {\em Acknowledgements.} We are grateful to Dick Canary, Francois Labourie, Feng Luo, Greg McShane, Hugo Parlier, Caroline Series,
and Ying Zhang for helpful
discussions on this material.

%long list of references, Bowditch, Sakuma etal, Tan Wong Zhang, Labourie-McShane, Bridgeman Kahn etc

%In this article we  consider four spectral identities for hyperbolic manifolds which are apriori different but which will show share a common framework. For this purpose, we will present slightly different proofs and perspectives than the original papers with the purpose of shedding some light on this underlying structure.
\section{Two orthospectra identities}
We let $M$ be a finite volume hyperbolic manifold with non-empty
totally geodesic boundary. In \cite{Bas93}, Basmajian introduced the
notion of  orthogeodesics for hyperbolic manifolds. An {\em
orthogeodesic} $\alpha$ for $M$ is an oriented proper geodesic arc
in $M$ which is perpendicular to $\partial M$ at its endpoints (see figure \ref{ortho}). Let
$O_M$ be the collection of orthogeodesics for $M$ and $L_M$ the set
of lengths of orthogeodesics (with multiplicity). The set $L_M$ is
the {\em ortholength spectrum}, note that all multiplicities are
even since we consider oriented orthogeodesics.

\begin{figure}[htbp] %  figure placement: here, top, bottom, or page
   \centering
   \includegraphics[width=4in]{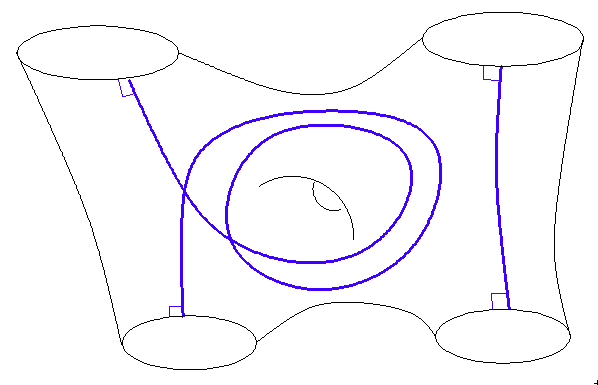} 
   \caption{Two orthogeodesics, one simple, one non-simple}
   \label{ortho}
\end{figure}

\subsection{ Basmajian Identity}
In the 1993 paper \cite{Bas93}, {\em The orthogonal spectrum of a hyperbolic manifold}, Basmajian derived the following orthospectrum identity:
\begin{theorem}{(Theorem A: Basmajian's Identity, \cite{Bas93})}
Let $M$ be a compact volume hyperbolic manifold with non-empty totally geodesic boundary and $V_k(r)$  the volume of the ball of radius $r$ in $\Hyp^k$. Then
$$\Vol(\partial M) = \sum_{l \in L_M} V_{n-1}\left(\log\left(\coth\frac{l}{2}\right)\right).$$
%where $V_k(r)$ is the volume of the ball or radius $r$ in $\Hyp^k$.
\end{theorem}

\subsection{ Bridgeman-Kahn identity}
In the 2011 paper \cite{BK10}, {\em Hyperbolic volume of n-manifolds with geodesic boundary and orthospectra}, Bridgeman and Kahn obtained  the following identity for the volume of the unit tangent bundle $T_1(M)$, again in terms of the ortholength spectrum $L_M$.
\begin{theorem}{(Theorem B: Bridgeman-Kahn Identity, \cite{BK10})}
Let  $M$ be a compact hyperbolic manifold with non-empty totally geodesic boundary, then
$$\Vol(T_1(M)) = \sum_{l \in L_M} F_{n}(l)$$
where  $F_n:\Real_+ \rightarrow \Real_+$ is an explicitly described smooth monotonically decreasing function depending only on the dimension $n$.
 \end{theorem}
We note that as $\Vol(T_1(M)) = \Vol(M).\Vol(\Sph^{n-1})$, the above identity can also be thought of as an identity for the hyperbolic volume of the manifold $M$.
\subsection{The surface case of Theorems A and B}
Theorems A and B are particularly interesting in the case of
hyperbolic surfaces as they give identities for deformation spaces
of Riemann surfaces with boundary. They also relate in this context
to the McShane and Luo-Tan identities which we describe in the next
section. The Bridgeman-Kahn identity in fact arose from a
generalization of a previous paper of the first named author
\cite{B11} which provided an explicit formula for the function
$F_2(l)$ in Theorem B in terms of the Roger's dilogarithm. We have:
\begin{theorem}{(Theorem $A'$: Basmajian identity for surfaces)}
Let $S$ be a hyperbolic surface with non-empty boundary $\partial S$. Then
$$\mbox{Length}(\partial S) = \sum_{l \in L_S} 2 \log\left(\coth\frac{l}{2}\right).$$
\end{theorem}
\begin{theorem}{(Theorem $B'$: Bridgeman identity)}
Let $S$ be a hyperbolic surface with non-empty boundary $\partial S$. Then
$$\Vol(T_1(S)) = 2\pi\mbox{Area}(S) = \sum_{l \in L_S} 4.{\mathcal R}\left(\sech^2\frac{l}{2}\right)$$
where ${\mathcal R}$ is the Rogers dilogarithm function.
\end{theorem}
The function ${\mathcal R}$ was introduced by Rogers in his 1907 paper \cite{Rog07}.  This function arises in hyperbolic volume calculations; the imaginary part of  ${\mathcal R}(z)$ is the volume of an ideal tetrahedron with vertices having cross-ratio $z$.
\subsection{Alternative derivations of the Bridgeman and Bridgeman-Kahn Identity}
In the paper \cite{Cal10a}, Calegari gives an alternate derivation
of the Bridgeman identity. Also in \cite{Cal10b}, Calegari derived
an orthospectrum  identity for all dimensions $\ge 2$ which arose
from a different decomposition of the unit tangent bundle, he showed in the surface case that
it is equal to the Bridgeman identity. In a recent preprint
\cite{MasMcS} by Masai and McShane, it was shown that for higher
dimensions, it is also equal to the Bridgeman-Kahn identity.

\section{Two Simple Spectra identities for hyperbolic surfaces}
Let $S$ be a finite area hyperbolic surface, we will consider
various cases including when $S$ has cusps, totally geodesic
boundary,  cone singularities (with cone angles $\le \pi$), and
finally, when $S$ is a closed surface. We saw already in the
previous section that when $S$ has non-empty totally geodesic
boundary, we can define the collection of orthogeodesics $O_S$,
which provided an index set for the Basmajian and Bridgeman
identities, which are then expressed in terms of the ortholength
spectrum $L_S$. This set can be extended in a natural way for
surfaces which also have cusps or cone singularities. For the
purposes of the next two classes of identities however, it is more
useful to consider the smaller collection $SO_S$ of {\em simple
orthogeodesics}, that is, orthogeodesics which do not have self
intersection, and also the collection $SG_S$ of {\em simple closed
geodesics} on $S$. We will  see that $SO_S$, together with collections of certain subsets of $SG_S$
consisting of one, two or three disjoint geodesics satisfying some
topological criteria  will be useful as
index sets for the identities.
\subsection{(Generalized) McShane Identity}
Let $S_{g,n}$ denote a hyperbolic surface of genus $g$ with $n$
cusps. In his 1991 thesis, {\em A remarkable identity for lengths of
curves}, McShane proved an identity for the lengths of simple closed
geodesics on any once-punctured hyperbolic torus $S_{1,1}$ which he
generalized later in \cite{McS98} to more general cusped hyperbolic
surfaces  $S_{g,n}$, $n \ge 1$.
\begin{theorem}{(Theorem C: McShane Identity, \cite{McS91}, \cite{McS98})}
\begin{enumerate}
\item If $S_{1,1}$ is a  hyperbolic torus with one cusp, then
$$\sum_{\gamma} \frac{1}{1+e^{l(\gamma)}} = \frac{1}{2}$$
where the sum is over all simple closed geodesics $\gamma$ in $S_{1,1}$.
\item If $S_{g,n}$ is a  hyperbolic surface of genus $g$ with $n$ cusps, where $n \ge 1$, then
$$\sum_{\gamma_1, \gamma_2} \frac{1}{1+exp(\frac{l(\gamma_1)+l(\gamma_2)}{2})} = \frac{1}{2}$$
where the sum is over all unordered pairs of simple closed geodesics
$\{\gamma_1, \gamma_2\}$ which bound together with a fixed cusp an
embedded pair of pants in $S_{g,n}$. Here we adopt the convention
that $\gamma_1$ or $\gamma_2$ may be one of the other cusps on
$S_{g,n}$, considered as a (degenerate) geodesic of length $0$.
\end{enumerate}
 \end{theorem}
The case of the punctured torus can be regarded as a special case of
the surface $S_{g,n}$ where in the sum, $\gamma=\gamma_1=\gamma_2$
since any simple closed geodesic $\gamma$ on $S_{1,1}$ cuts it into
a pair of pants.

 In her 2005 thesis {\em Simple geodesics and Weil-Petersson volumes of moduli spaces of bordered Riemann surfaces},
 see \cite{Mir05, Mir07}, Mirzakhani gave a general version of the McShane identity for hyperbolic surfaces with
 geodesic boundaries and cusps, which was an important tool for her computation of the Weil-Petersson volumes of
 the moduli spaces. Independently, Ying Zhang in his 2004 thesis {\em Hyperbolic cone surfaces, generalized
 Markoff Maps, Schottky groups and McShane's identity}, see \cite{Zhang04, TWZ06}, also gave a generalization
 of the McShane identity for surfaces, with (non-empty) boundary consisting of cusps, totally geodesic boundaries,
 or cone singularities with cone angle $\le \pi$, with slightly different forms for the functions involved in the
 identity.
We first state Mirzakhani's generalization  and explain how to interprete the identity for cone surfaces later.
\begin{theorem}{(Theorem $C'$: Generalized McShane-Mirzakhani Identity, \cite{Mir05, Mir07})}
Let $S$ be a finite area hyperbolic surface with geodesic boundary components $\beta_1,\ldots,\beta_n$, of length $L_1,\ldots,L_n$.
Then
$$\sum_{\{\gamma_1, \gamma_2\}} D(L_1, l(\gamma_1), l(\gamma_2)) + \sum_{i=2}^n \sum_{\gamma} R(L_1,L_i,l(\gamma)) = L_1$$
where the first sum is over all unordered pairs of (interior) simple
closed geodesics bounding a pair of pants with $\beta_1$, and the
second sum is over simple closed geodesics bounding a pair of pants
with $\beta_1, \beta_i$ and
$$D(x,y,z) = 2\log\left(\frac{e^{\frac{x}{2}}+e^{\frac{y+z}{2}}}{e^{-\frac{x}{2}}+e^{\frac{y+z}{2}}}\right), ~~ R(x,y,z) = x- \log\left(\frac{\cosh(\frac{y}{2})+\cosh(\frac{x+z}{2})}{\cosh(\frac{y}{2})+\cosh(\frac{x-z}{2})}\right).$$
 \end{theorem}
The case where some of the other boundaries are cusps but $\beta_1$ is a geodesic boundary can be deduced from the above, by considering cusps to be boundaries of length $0$, where again we adopt the convention that a cusp may be regarded as a geodesic of length $0$ in the summands above. The case where $\beta_1$ is also a cusp, that is, $L_1=0$ is more interesting, in this case, the original McShane identities can be deduced from the above by taking the infinitesimal of the limit as $L_1 \rightarrow 0$,  or taking the formal derivative of the above identity with respect to $L_1$ and evaluating at $L_1=0$.
More interestingly, a cone singularity of cone angle $\theta$ may be regarded as a boundary component with purely imaginary complex length $i\theta$, and the above identity is also valid if some of the boundary components are cone singularities of cone angles $\le \pi$ as shown in \cite{Zhang04, TWZ06}. The restriction to cone angles $\le \pi$ is necessary in the argument, this guarantees a convexity property and the existence of geodesic representatives for essential simple closed curves on the surface. For example, if $S$ is a surface of genus $g>1$ with one cone singularity $\Delta$ of cone angle $\theta$, then we have
$$\sum_{\{\gamma_1, \gamma_2\}} D(i\theta, l(\gamma_1), l(\gamma_2))   = i\theta$$
where $\{\gamma_1, \gamma_2\}$ are unordered pairs of simple closed
geodesics bounding a pair of pants with $\Delta$.
We note that each of the summands of the identity in Theorem  $C'$ is
the measure of some subset $X_i$ of $\beta_1$: the summands in
the first sum of Theorem $C'$ correspond to sets which are indexed
by (a subset of) the simple orthogeodesics from $\beta_1$ to itself, the summands in
the second sum correspond to sets which are indexed by simple
orthogeodesics from $\beta_1$ to $\beta_i$ as we will see in the
proof later.
\subsection{ Luo-Tan identity}
The Basmajian, Bridgeman and McShane identities for surfaces were in
general only valid for surfaces with boundary; in the first two
cases, for surfaces with geodesic boundary, in the third case, to
surfaces with at least some cusp or cone singularity. They do not
extend to general closed surfaces without boundary. However, for the
genus $2$ hyperbolic surface $S_2$, by considering $S_2/hyp$ where
$hyp$ is the hyper-elliptic involution on $S_2$, one may lift the
identity on the cone surface $S_2/hyp$ to  obtain identities on the
closed genus two surface $S_2$, see \cite{McS06, TWZ06}. This method
however does not generalize to higher genus. In their 2011 preprint,
{\em A dilogarithm identity on moduli spaces of curves}, F. Luo and
the second author derived the following identity for closed
hyperbolic surfaces.
\begin{theorem}{(Theorem $D$: Luo-Tan identity, \cite{LT11})}
Let $S$ be a closed hyperbolic surface. There exist functions $f$ and $g$ involving the dilogarithm of the lengths of the simple geodesic loops in a 3-holed sphere or 1-holed torus, such that
$$\Vol(T_1(S)) = \sum_{P}f(P) + \sum_{T} g(T) $$
where the first sum is over all properly embedded 3-holed spheres $P
\subset S$ with geodesic boundary, the second sum is over all
properly embedded 1-holed tori $T \subset S$ with geodesic boundary.
\end{theorem}
The functions $f$ and $g$ are defined on the moduli
spaces of simple hyperbolic surfaces ($3$-holed spheres and
$1$-holed tori) with geodesic boundary and given in terms of  ${\mathcal R}$,  the Rogers dilogarithm function as follows:

Suppose $P$ is a hyperbolic 3-holed sphere with geodesic
boundaries of lengths $l_1, l_2, l_3$. Let $m_i$ be the length of
the shortest path from the $l_{i+1}$-th boundary to the
$l_{i+2}$-th boundary ($l_4=l_1$, $l_5=l_2$). Then
{\small \begin{equation}\label{eqn:deffp}
f(P): =4 \sum_{i \neq j}\left[2{\mathcal R}\left(\frac{1-x_i}{1-x_i y_j}\right)
-2{\mathcal R}\left(\frac{1-y_j}{1-x_i y_j}\right)-{\mathcal R}(y_j)
-{\mathcal R}\left(\frac{(1-x_i)^2y_j}{(1-y_j)^2 x_i}\right)\right]
\end{equation}}
where  $x_i = e^{-l_i}$ and $y_i=\tanh^2(m_i/2)$.

%Note that the sine and cosine rules imply that $m_i$ can be
%expressed in terms of $l_j$'s.

Suppose $T$ is a hyperbolic 1-holed torus with geodesic boundary.
 For any non-boundary parallel simple closed
geodesic $A$ of length $a$ in $T$, let $m_{A}$ be the distance
between $\partial T$ and $A$. Then $g(T):=$ 
{\small \begin{equation}\label{eqn:definitionofg}
  4\pi^2+8\sum_{A} \left[ 2{\mathcal R}\left(\frac{1-x_A}{1-x_Ay_A}\right) - 2{\mathcal R}\left(\frac{1-y_A}{1-x_Ay_A}\right)
-2 {\mathcal R}(y_A) -{\mathcal R}\left(\frac{(1-x_A)^2y_A}{(1-y_A)^2 x_A}\right)\right]
\end{equation}}
where $x_A = e^{-a}$ and $y_A=\tanh^2(m_A/2)$ and the sum is over
all non-boundary parallel simple closed geodesics $A$ in $T$.

\subsection{Luo-Tan identity for surfaces with boundary and non-orientable
surfaces} The Luo-Tan identity also holds for surfaces $S$ with
geodesic boundary and cusps, however, in this case, the functions $f$ and $g$ need to
be modified when $P$ or $T$ share some boundary component with $S$, similar to the  McShane-Mirzakhani identity for surfaces
with more than one boundary component, see \cite{ LT13}. In this case, the identity is trivial for the one-holed torus -  however,
one can obtain a meaningful identity involving the lengths of the
simple closed geodesics in $T$  by a topological covering
argument, using the identity for a four-holed sphere, see \cite{H13, HT13}. For example, for a once-punctured hyperbolic torus $T$, we obtain
\begin{equation}\label{eqn:LuoTanfortorus}
\sum_{\gamma} \left[{\mathcal R}(\sech^2(l(\gamma)/2))+2\left({\mathcal R}\left(\frac{1+e^{-l(\gamma)}}{2}\right)-{\mathcal R}\left(\frac{1-e^{-l(\gamma)}}{2}\right)\right)\right]=\frac{\pi^2}{2}
\end{equation} where the sum extends over all simple closed geodesics $\gamma$ in $T$.

 For non-orientable surfaces, one also
obtain an analogous identity, in this case, the summands include
terms coming from embedded simple non-orientable surfaces, namely,
one holed Klein bottles and one-holed M\"obius bands, see \cite{LT13}.

\section{Proofs of Theorems A and C - Boundary flow}
As remarked in the introduction, the proofs of all the results will
be based on a decomposition of certain sets $X$ associated to $M$.
In particular for Theorems A and C, $X$  will be subsets of $T_1(M)$
associated to the boundary $\partial M$ and  the proofs arise from consideration of
the boundary flow on $X$.
\subsection{Proof of Basmajian Identity}
Let $T_1(M)$ be the unit tangent bundle of $M$ and $\pi: T_1(M)
\rightarrow M$ be the projection map. For the Basmajian identity, we
let $X$ be the set of unit tangent vectors $v$ whose basepoint are
on $\partial M$, and which are perpendicular to $\partial M$, and
point into the interior of $M$, that is,
 $$X=\{v \in T_1(M)~:~ \pi(v) \in \partial M, v \perp \partial M, v ~\mbox{points into}~ M\}.$$
Clearly, $X$ identifies with $\partial M$ under $\pi$, and we define  the measure $\mu$ on $X$ to be the pullback of Lebesgue measure on $\partial M$ under $\pi$. In
particular, $\mu(X)=\mbox{Vol}(\partial M)$. We
consider the unit speed geodesic $g_v$ starting at $p=\pi(v) \in
\partial M$ obtained by exponentiating $v$. Thus $g_v$  is the geodesic arc
obtained by flowing from $\pi(v)$ until you hit the boundary. To derive the
Basmajian identity, we let $$Z=\{v \in X ~:~ {\mbox{Length}}(g_v)=\infty \}.$$
 It follows from the fact that the limit set of $M$ is measure zero,
that $Z$ has zero volume. For each of the remaining vectors, $g_v$
is a geodesic arc of finite length with endpoints on $\partial M$.
We define an equivalence relation on these by defining $v \sim w$ if
$g_v, g_w$ are homotopic rel. boundary in $M$. Then each oriented
orthogeodesic $\alpha$ defines an equivalence class $E_\alpha$ given
by
$$E_\alpha = \{ v \in X-Z\ |\ g_v \mbox{ is homotopic rel boundary to } \alpha\}.$$
Also, by lifting to the universal cover and using a tightening argument, we see that for every $v
\in X \setminus Z$,  $g_v$ is homotopic rel. boundary to a
orthogeodesic $\alpha$, so
$\{E_\alpha\}_{\alpha \in O_M}$ covers $X \setminus Z$. Furthermore
if $\alpha \neq \beta$ then $E_\alpha \cap E_\beta =\emptyset$ as no
two orthogeodesics are homotopic rel boundary. Thus we have the
partition  $ X= Z \sqcup \bigsqcup_{\alpha} E_\alpha$
and the associated identity
$$\Vol(X) = \Vol(\partial M)=\sum_{\alpha \in O_M} \Vol(E_\alpha)$$
To calculate $\Vol(E_\alpha)$, we lift to the universal cover so
that $\alpha$ lifts to a geodesic arc $\tilde \alpha$ orthogonal to
two boundary hyperplanes $P, Q$. As $\alpha$ is oriented, we assume
$\tilde\alpha$ is oriented from $P$ to $Q$. Thus any $g_v$ homotopic
rel boundary to $\alpha$ has a unique lift $\tilde g_v$ which is a
geodesic arc perpendicular to $P$ going from $P$ to $Q$. Hence
$\tilde g_v$ has basepoint in the orthogonal projection of $Q$ onto
$P$. Thus the set $\pi(E_\alpha)$ lifts to a disk of radius
$r(\alpha)$ given by orthogonal projection of $Q$ onto $P$. Let
$r(\alpha), l(\alpha)$ be the two finite sides of a hyperbolic
quadrilateral with one ideal vertex and finite angles $\pi/2$ (see
figure \ref{basmajian1}). By elementary hyperbolic geometry (see
\cite{Bear95}, for example), $\sinh(r(\alpha)).\sinh(l(\alpha)) =
1$, giving $r(\alpha) = \log(\coth(l(\alpha)/2))$ and  we obtain
$$\Vol(\partial M) = \sum_{\alpha \in O_M} \Vol(E_\alpha) = \sum_{\alpha \in O_M} V_{n-1}(\log(\coth(l(\alpha)/2))).$$
\begin{figure}[htbp] %  figure placement: here, top, bottom, or page
   \centering
   \includegraphics[width=3.5in]{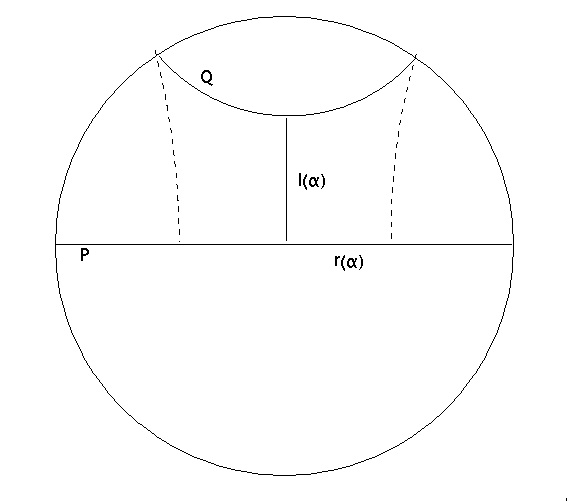}
   \caption{Orthogonal Projection onto a plane}
   \label{basmajian1}
\end{figure}
\noindent{Remark}: The identity generalizes to hyperbolic manifolds
with cusps, as long as the boundary contains some non-empty geodesic
component. This is necessary to deduce that $\mu(Z)=0$.
Othogeodesics ending in a cusp have infinite length and do not
contribute to the summands of the identity.
\subsection{Proof of the Generalized McShane Identity}
We consider a hyperbolic surface $S$ with a finite number of
geodesic boundary components, cusps and cone singularities (with
cone angles $\le \pi$).

For simplicity we first consider the case where $\partial S$ has
only geodesic components $\beta_1, \ldots, \beta_n$ with lengths
$L_1, \ldots, L_n$, the basic idea of the proof is the same for the
more general case, we will explain later how to modify the proof if
some of the $\beta_i$'s are  cusps or cone singularities of cone
angle $\le \pi$. We derive the identity based at $\beta_1$, as such,
define
$$X=\{v \in T_1(S)~:~ \pi(v) \in \beta_1, v \perp \beta_1, v ~\mbox{points into}~ S\}.$$
Clearly, $\pi$ induces a bijection from $X$ to $\beta_1$, so
$\mu(X)= \mbox{Length}(\beta_1)=L_1$. Again, as in the proof of the
Basmajian identity, we are going to consider the unit speed geodesic
obtained by exponentiating $v \in X$, however, this time, we are
going to stop when the geodesic hits itself, or the boundary
$\partial S$. More precisely, let $G_v:[0,T] \rightarrow S$ be the
geodesic arc obtained by exponentiating $v \in X$ such that $G_v$ is
injective on $[0,T)$ and either $G_v(T)=G_v(s)$ for some $s\in
[0,T)$, or $G_v(T)\in \partial S$. If $G_v[0,t]$ is defined and
injective for all $t>0$, then  $T=\infty$, that is $G_v$ is a simple
geodesic arc of infinite length. A good  analogy for the
difference between $g_v$ in the Basmajian proof and
$G_v$ in the McShane proof is that we should consider
$g_v$ as a laser beam starting from $\partial M$ which is
allowed to intersect itself any number of times, until it hits the
boundary, whereas $G_v$ should be thought of as a  wall,
which terminates when it hits itself, or the boundary.

Now let $Z \subset X$ be the set of vectors for which $G_v$ has
infinite length, that is, $T=\infty$. Again, by the same argument as
before, $\mu(Z)=0$ since the limit set of $S$ has measure zero, and
the endpoints of the lifts of $G_v$ must land on the limit set if $v
\in Z$. However, if $\beta$ is a cusp, then we need a stronger
result, namely, the Birman-Series result \cite{BS85} that the set
of simple geodesics on the surface has Hausdorff dimension 1, which
implies that $\mu(Z)=0$. Similarly, if $\beta$ is a cone point, 
we  require a generalization of the Birman-Series result, see
\cite{TWZ06}. We note that $Z$ has a rather complicated Cantor set
structure, and McShane analysed this set carefully in \cite{McS98}.
However, for the purposes of proving the identity, the structure of
$Z$ is irrelevant, and one only really needs to know that
$\mu(Z)=0$.

We now look at $G_v$ for $v \in X \setminus Z$, in this case $G_v$
is either a finite geodesic arc ending in a loop (a lasso), or a simple
geodesic arc from $\beta_1$ to $\partial S$. \begin{itemize}
\item
If $G_v$ is a lasso, or a simple arc ending in $\beta_1$, then a
regular neighborhood $N$ of $\beta_1 \cup G_v$ in $S$ is
topologically a pair of pants, where one of the boundary components
is $\beta_1$. The other two boundary components can then be {\em
tightened} to geodesics $\gamma_1, \gamma_2$ which are {\em
disjoint} and which bound together with $\beta_1$ an embedded pair
of pants in $S$ which {\em contains}  $G_v$ (note that
in the case where $S$ is a one-holed torus, then
$\gamma_1=\gamma_2:=\gamma$,  where $\gamma$ is a simple closed
geodesic on $S$ disjoint from $G_v$, otherwise, $\gamma_1$ and
$\gamma_2$ are distinct and disjoint).

\item If $G_v$ is a simple arc from $\beta_1$ to $\beta_i$ where $i
\neq 1$, then a regular neighborhood of $\beta_1 \cup G_v \cup
\beta_i$ is again a pair of pants, where $\beta_1$ and $\beta_i$ are
two of the boundary components. The third boundary  can again be
tightened to a simple closed geodesic $\gamma$, and again $G_v$ is
contained in the resulting pair of pants.
\end{itemize}

\medskip

One can prove the assertions in the previous paragraph by a cut and
paste argument as follows: Cut
$S$ along $G_v$ to obtain a (not necessarily connected) convex
hyperbolic surface $S_{cut}$ with either  two piece-wise geodesic
boundaries (if $G_v$ is a lasso or simple arc from $\beta_1$ to
itself), or one piece-wise geodesic boundary (if $G_v$ is a simple
arc from $\beta_1$ to $\beta_i$, $i \neq 1$), and other geodesic
boundaries. Note that if $S$ is a one-holed torus then $S_{cut}$ is a cylinder
whose core is a geodesic $\gamma$ disjoint from $\partial S_{cut}$.
Otherwise, in the first case, let $\gamma_1$ and $\gamma_2$ be the
two disjoint geodesics which bound the convex core of $S_{cut}$
which again are disjoint from $\partial S_{cut}$. Regluing along
$G_v$, we see that $\gamma_1$ and $\gamma_2$ bound together with
$\beta_1$ a pair of pants in $S$ (basically the complement of the
convex core of $S_{cut}$) which contains $G_v$, as asserted. The
same argument applies to the second case to obtain a pair of pants
bounded by $\beta_1, \beta_i$ and a geodesic $\gamma$.

\medskip

To recap, for every $v\in X \setminus Z$, $G_v$ is a geodesic arc
contained in a {\em unique} pair of pants embedded in $S$ bounded by
$\beta_1$ and a pair of geodesics $\gamma_1$, $\gamma_2$ (where one of
$\gamma_1, \gamma_2$ may be a different boundary component $\beta_i$ of $S$).
Let $\mathcal P$ be the set of all such pairs of pants
embedded in $S$ (equivalently, all unordered pairs of geodesics
$\{\gamma_1, \gamma_2\}$ in $S$ which bound a pair of pants with
$\beta_1$), and for each $P \in {\mathcal P}$, we define
$$X_P=\{v \in X \setminus Z ~:~ G_v \subset P\},$$ then  $X$
is the disjoint union of $Z$ and the $X_P$'s, hence
$$L_1=\sum _{P\in {\mathcal P}} \mu(X_P).$$

\bigskip
To derive the formulae for $\mu(X_P)$ we consider a general pair of
pants $P$ with boundary $a,b,c$ of lengths $x,y,z$. We consider
perpendicular geodesics $\alpha_p$ from points $p$ on $a$. We let
$p_1,p_2$ be the two points on $a$ such that $\alpha_{p_i}$ is a
simple geodesic spiraling towards $b$ (one for each direction) and
let $q_1,q_2$ be the two points on $a$ such that $\alpha_{q_i}$ is a
simple geodesic spiraling towards $c$. Also assume that
$p_1,q_1,q_2,p_2$ is the cyclic ordering of the points on $a$, which
divide $a$ into the  intervals $[p_1,q_1], [q_1, q_2], [q_2,p_2]$
and $[p_2,p_1]$ with disjoint interiors (see figure \ref{pants}). Each interval contains a
unique point $m$ such that $\alpha_m$ is a {\em simple
orthogeodesic}
 from $a$ to $a, c, a$, and $b$ respectively. We see that
for $p\in (p_2,p_1)$, $\alpha_p$ is a simple geodesic from $a$ to
$b$, similarly, for $p\in (q_1,q_2)$, $\alpha_p$ is a simple
geodesic from $a$ to $c$. For $p \in (p_1,q_1) \cup (q_2,p_2)$,
$\alpha_p$ is either a simple geodesic from $a$ to itself or it has
self-intersections. Furthermore, $L[p_2,p_1]$ is precisely the
orthogonal projection of $b$ to $a$, similarly  $L[q_1,q_2]$ is the
orthogonal projection of $c$ to $a$, and $L[p_1,q_1]=L[q_2,p_2]$ by
symmetry.

\begin{figure}[htbp] %  figure placement: here, top, bottom, or page
   \centering
   \includegraphics[width=4in]{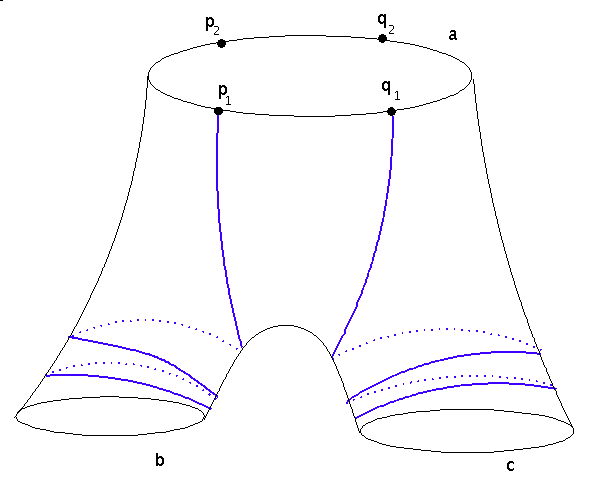} 
   \caption{Pants P with spiraling geodesics}
   \label{pants}
\end{figure}

Now let $P \in {\mathcal P}$ where $a=\beta_1$. We have:
\begin{itemize}
    \item If $b=\gamma_1$
and $c=\gamma_2$ are interior curves of $S$, then $G_v \subset X_P$ if
and only if $\pi(v) \in (p_1,q_1) \cup (q_2,p_2)$. Then
$\mu(X_P)=D(x,y,z)=2L[p_1,q_1]$. By elementary hyperbolic geometry
we have
$$D(x,y,z) = 2\log\left(\frac{e^{\frac{x}{2}}+e^{\frac{y+z}{2}}}{e^{-\frac{x}{2}}+e^{\frac{y+z}{2}}}\right),$$

    \item If say $b=\gamma$
is an interior geodesic and $c=\beta_i$ is a boundary geodesic of
$S$, then $G_v \in X_P$ if and only if $\pi(v) \in (p_1,q_1) \cup
(q_1,q_2) \cup(q_2,p_2)$. We have $\Length[p_2,p_1] =
\log(\coth(Z/2))$ where $Z$ is the length of the perpendicular arc
from $a$ to $b$. Applying the hyperbolic cosine rule we get
$$R(x,y,z) = x- \log\left(\frac{\cosh(\frac{y}{2})+\cosh(\frac{x+z}{2})}{\cosh(\frac{y}{2})+\cosh(\frac{x-z}{2})}\right).$$
\end{itemize}

The generalized McShane identity now follows by substitution.

%
% . Let $[p_1,p_2]$ denote
%the interval on $a$ that does not contain either $q_i$ and similarly
%for interval $[q_1,q_2]$. We see that for $p \in [p_1, p_2]$,
%$\alpha_p$ is a simple geodesic arc with endpoints in $a,b$ and
%similarly for $[q_1,q_2]$. If $p \in [p_1,q_1] \cup[p_2,q_2]$ then
%$\alpha_p$ intersects itself in $P$ or intersects boundary curve $a$
%at another point. Let $P \in {\mathcal P}$. If $P$ has boundary
%curves $a = \beta_1, b = \gamma, c = \beta_i$, then for $x \in
%(p_1,p_2)$, $\alpha_x$ intersects $\gamma$ and therefore $P_x \neq
%P$. Otherwise if $p \not\in[p_1,p_2]$ we have $P_x = P$. Thus we
%define $R(x,y,z) = \mbox{Length}([p_1,p_2]^c) = x -
%\mbox{Length}([p_1,p_2])$ and have that $\mu(E_P) = R(x,y,z)$. The
%length of $[p_1,p_2]$ is exactly the length of the orthogonal
%projection of $b$ onto $a$. Thus $\Length[p_1,p_2] =
%\log(\coth(Z/2))$ where $Z$ is the length of the perpendicular arc
%from $a$ to $b$. Applying the hyperbolic cosine rule we get
%$$R(x,y,z) = x- \log\left(\frac{\cosh(\frac{y}{2})+\cosh(\frac{x+z}{2})}{\cosh(\frac{y}{2})+\cosh(\frac{x-z}{2})}\right).$$
%If $P$ has boundary curves $a = \beta_1, b = \gamma_1 ,c = \gamma_2$
%then for $x \in [p_1,q_1]\cup[p_2,q_2]$ we have $P_x = P$. We
%therefore we have  $\mu(E_P) = D(x,y,z) =
%\Length([p_1,q_1]\cup[p_2,q_2]) = 2\Length([p_1,q_1])$. By
%elementary hyperbolic geometry we have
%$$D(x,y,z) = 2\log\left(\frac{e^{\frac{x}{2}}+e^{\frac{y+z}{2}}}{e^{-\frac{x}{2}}+e^{\frac{y+z}{2}}}\right)$$
%Thus we obtain the generalized McShane identity by substitution.

\subsection{Cusps and cone points}
If $\beta_1$ is a cusp, we can take a  horocycle $C$ of length 1
about the cusp and remove the neighbourbood of the cusp bounded by
$C$. We now take $X$ to be the set of vectors in $T_1(S)$ such that
$\pi(v) \in C$ and $v$ is perpendicular to $C$ and pointing into
$S$. Essentially the same analysis works to give a decomposition of
$X$, with suitable modifications of the functions $D$ and $R$, which
now only depend on two variables. In the tightening argument, we
need the fact that the horocycle chosen is sufficiently small so
that it is dijoint from all simple closed geodesics, choosing length
one as we did works. If all other boundaries are cusps, then we
recover the original McShane identity since in this case,
$$D(y,z)=R(y,z)=\frac{1}{1+e^{\frac{y+z}{2}}}.$$ If $\beta_1$ is a cone
point of cone angle $\theta_1$, we decompose the set of tangent
vectors based at the cone point again in essentially the same way. We note that in order for the suface obtained after cutting to be convex, the restriction that the cone angle $\le \pi$ is necessary, we need this to perform the tightening argument. Similarly, we require all other cone angles to be $\le \pi$ if we want every essential simple closed curve to be represented by a geodesic (or the double cover of a geodesic segment between two cone points of angle $\pi$). 

Here it is useful to regard a cone point as an axis through the point perpendicular to the
hyperbolic plane, and  use the complex measure of length between two
skew axes in $\Hyp^3$. With this, the measure of the angle is purely
imaginary. Similarly, other components $\beta_i$ which are cone
angles should be regarded as axes perpendicular to the plane, and we
recover exactly the same identity as that obtained by Mirzakhani,
with the same functions, with the convention that cone points have
purely imaginary lengths, see \cite{TWZ06} for details.

\subsection{Index sets and the relation with Basmajian identity}
The set $X$ in the McShane identity is decomposed into the disjoint
union of $Z$, a set of measure 0 and a countable union of disjoint
open intervals  $X_{\alpha}$, which from the previous observation is
indexed by $\alpha \in SO_S(\beta_1)$, the set of {\it simple} orthogonal
geodesics on $S$ with base point on $\beta_1$. Each such simple orthogeodesic
gives rise to an interval in $\beta_1$, all of which are disjoint. The first sum consists of summands corresponding to the (two) intervals from the  simple orthogeodesics from $\beta_1$ to itself contained in the pants $P$ where $\gamma_1$ and $\gamma_2$ are interior geodesics, a summand of the second sum consists of three intervals, the extra interval coming from the simple orthogeodesic contained in $P$ from $\beta_1$ to $\beta_i$. 

If $\alpha \in SO_S(\beta_1)$ goes from $\beta_1$ to itself, then
$\mu(X_{\alpha})=L[p_1,q_1]=D(x,y,z)/2$ (note that this length
depends on the geometry of the pants $P$ and not just the length of
$\alpha$) and if $\alpha$ is a simple orthogonal geodesic from
$\beta_1$ to another component $\beta_i$, then $\mu(X_{\alpha})=2
\log\left(\coth\frac{l(\alpha)}{2}\right)$, the projection of
$\beta_i$ to $\beta_1$ along $\alpha$.

The index set for the Basmajian identity is much larger, and
strictly contains the index set for the McShane identity. In this
sense, the Basmajian identity for surfaces, as restricted to
$\beta_1$ is a refinement of the McShane identity: the terms
corresponding to simple geodesics from $\beta_1$ to a different
component $\beta_i$ are the same for both identities, however, in the McShane identity, each
set $X_{\alpha}$ where $\alpha$ is a simple geodesic from $\beta_1$
to itself  contains infinitely many terms in
the Basmajian identities, as infinitely many non-simple
orthogeodesics have the non-intersecting beginning part (the
geodesic segment $G_v$ defined earlier) contained in the same pants
$P$.

We note also that the index set for the McShane identity can be regarded as the set of all embedded pairs of pants in $S$ which contain $\beta_1$ as a boundary. These in turn split into two subsets, pairs of pants $P$ for which $\partial P \cap \partial S =\beta_1$ or $\partial P \cap \partial S=\beta_1 \cup \beta_i$ for some $i \neq 1$. The first type gives the first sum, the second type the second sum in Theorem $B'$. This point of view is useful as it generalizes to the Luo-Tan identity.

\section{Proofs of Theorems B and D: Interior Flow}
For the Bridgeman-Kahn and Luo-Tan identities, we consider $M$ a
hyperbolic manifold and $T_1(M)$ its unit tangent bundle. We let $X
= T_1(M)$ with $\mu$  the volume measure on $T_1(X)$. We then consider for
each $v \in T_1(M)$ the geodesic obtained by flowing in both
directions. We will show that the two identities described are
obtained by considering the dynamical properties of this geodesic.

\subsection{Proof of the Bridgeman-Kahn identity}
Let $M$ be a hyperbolic manifold with totally geodesic boundary. We
let $X = T_1(M)$ and $\mu$ be the volume measure on $T_1(M)$. For each
$v \in T_1(M)$ we let $g_v$ be the maximal geodesic arc tangent to
$v$. To derive the Bridgeman-Kahn identity, we let $Z$ be the set of
$v$ such that $g_v$ is not a proper geodesic arc (i.e. the flow does
not hit the boundary in at least one direction). Once again, as the
limit set of $M$ is measure zero, the set $Z$ satisfies $\mu(Z) =
0$. For $v \not\in Z$ we have that $g_v$ is a proper geodesic arc,
and as in the Basmajian identity, we define an equivalence relation
by $v \sim w$ if $g_v, g_w$ are homotopic rel boundary. Once again,
each orthogeodesic $\alpha$ defines an equivalence class $E_\alpha$
and as before, they form a partition. Thus we have the associated
identity
$$\Vol(T_1(M)) = \sum_{\alpha \in O_M} \Vol(E_\alpha)$$
To calculate $\Vol(E_\alpha)$, we  lift $\alpha$ to the universal
cover such that it is perpendicular to two boundary planes $P,Q$. By
definition, any $v \in E_\alpha$ has $g_v$ homotopic rel boundary to
$\alpha$. Thus $g_v$ has a unique lift to a geodesic $\tilde g_v$
which has endpoints on $P,Q$. Hence $E_\alpha$ also has a unique
lift to $\tilde E_\alpha$ where
$$\tilde E_\alpha = \left\{ v\in T_1(\Hyp^n)\ | \ \exists\  a \leq 0, b \geq 0 \mbox{  such that } \tilde g_v(a) \in P, \tilde g_v(b) \in Q\right\}.$$
The volume of this set only depends on $l(\alpha) = d(P,Q)$ and the
dimension. Therefore we have
$$\Vol(E_\alpha) = F_n(l(\alpha))$$
for some function $F_n$ which gives the Bridgeman-Kahn identity. To
derive a formula for $F_n$, we let $\Omega$ be the volume measure on
the unit tangent bundle to the upper half space model for $\Hyp^n$,
invariant under $Isom(\Hyp^n)$. We let $G(\Hyp^n)$ be the space of
oriented geodesics in $\Hyp^n$ and identify $G(\Hyp^n) =
(\Sph^{n-1}_\infty\times\Sph^{n-1}_\infty \setminus\mbox{Diagonal})$
by assigning to $g$ the pair of endpoints $(x,y)$. We have a natural
fiber bundle $p:T_1(\Hyp^n)\rightarrow G(\Hyp^n)$ by letting  $p(v)$
be the oriented geodesic tangent to $v$. We obtain a parametrization
of $T_1(\Hyp^n)$ by choosing a basepoint on each geodesic. We let $b
\in \Hyp^n$ and for each geodesic $g$, let $b_g$ be the nearest
point of $g$ to $b$. Then to each $v \in T_1(\Hyp^n)$ we assign the
triple $(x,y,t) \in \Sph^{n-1}_\infty\times \Sph^{n-1}_\infty \times
\Real$ where $(x,y) = p(v)$  and $t$ is the signed hyperbolic
distance along the geodesic $p(v)$ from $b_{p(v)}$ to  $v$.  In
terms of this parametrization,
$$d\Omega = \frac{2dV_x dV_ydt}{|x-y|^{2n-2}}$$
where $dV_x = dx_1dx_2\ldots,dx_n$ and $|x-y|$ is the Euclidean
distance between $x,y$ (see \cite{Nic89}). We choose planes $P, Q$ such that $d(P,Q) =
l$ to be given by the planes intersecting the boundary in the
circles of radius $1, e^l$ about the origin. We define
$$E_l =  \left\{ v\in T_1(\Hyp^n)\ | \ \exists\  a \leq 0, b \geq 0 \mbox{  such that } \alpha_v(a) \in P, \alpha_v(b) \in Q\right\}.$$
If $g = (x,y)$ is a geodesic intersecting planes both $P,Q,$ we let
$L(x,y,l) = d(P\cap g, Q \cap g)$, the length between intersection
points. Alternately we have $L(x,y,l) = \Length(p^{-1}(g) \cap
E_l)$. Then integrating over $t$ we have
$$F_n(l) = \int_{E_l} d\Omega = \int_{g \in p(E_l)} \left( \int_{p^{-1}(g) \cap E_l} dt\right)\frac{2 dV_x dV_y}{|x-y|^{2n-2}}$$
giving
\begin{equation}
F_n(l) = \int_{|x| < 1} \int_{|y| > e^l}  \frac{2L(x,y,l)dV_x dV_y}{|x-y|^{2n-2}}.
\label{bk_int_form}
\end{equation}
This integral formula can be simplified to obtain a closed form in
even dimensions and can be reduced to an integral over the unit
interval of a closed form in odd dimensions, in particular, when
$n=2$, it takes on the explicit form given in the Bridgeman identity
which we describe in the next subsection.

\subsection{Dilogarithm Identities} We first describe the Rogers
dilogarithm. We define the $k$th polylogarithm function $Li_k$ to be
the analytic function with Taylor series
$$Li_k(z) = \sum_{n=1}^\infty \frac{z^n}{n^k}\qquad \mbox{ for } |z| < 1.$$
Then we have
$$Li_0(z) = \frac{z}{1-z} \qquad Li_1(z) = -\log(1-z).$$
Also they satisfy the recursive formula
$$Li_k'(z) = \frac{Li_{k-1}(z)}{z}.$$
The function $Li_2$ is the dilogarithm function and Rogers
dilogarithm function is a normalization of it given by
$${\mathcal R}(z) = Li_2(z) + \frac{1}{2}\log|z|\log(1-z).$$
In the paper  \cite{B11}, the first author considered the surface case of the Bridgeman-Kahn identity. We let $S$ be a  finite area surface with totally geodesic boundary. Then boundary components of $S$ are either closed geodesics or bi-infinite geodesics (such as the case when $S$ is an ideal polygon). A {\em boundary cusp} of $S$ is a cusp on the boundary of $S$ bounded by two bi-infinite geodesics contained in the boundary of $S$.  Let $N_S$ be the number of boundary cusps. Then we have the following generalized version 
\begin{theorem}{(Bridgeman, \cite{B11})}
Let $S$ be a finite area hyperbolic surface with totally geodesic boundary. Then
$$\Vol(T_1(S)) = 2\pi \mbox{Area}(S) = \sum_{l \in L_S} 4.{\mathcal R}\left(\sech^2\frac{l}{2}\right) + \frac{2\pi^2}{3}N_S.$$
where ${\mathcal R}$ is the Rogers dilogarithm function.
\end{theorem}

In order to prove this, we once again take $X = T_1(S)$ and $\mu$ volume measure. We again define $Z$ to be the set of $v$ such $g_v$ is not a proper geodesic arc. Then if $v \not\in Z$, $g_v$ is a proper geodesic arc and we define $v \sim w$ if $g_v, g_w$ are homotopic rel boundary. There are two cases, either $g_v$ is homotopic rel boundary to an orthogeodesic $\alpha$ or $g_v$ is homotopic to a neighborhood of a boundary cusp $c$. Thus we have equivalence classes $E_\alpha$ for each orthogeodesic and $E_c$ for each boundary cusp. Therefore
$$\Vol(T_1(S)) = \sum_{\alpha} \Vol(E_\alpha) + \sum_{c_i} \Vol(E_{c_i}).$$
We have by definition that $\Vol(E_\alpha) = F_2(l(\alpha))$. Also, the sets $E_{c_i}$ are all isometric. Therefore
$$\Vol(T_1(S)) = \sum_{l \in L_S} F_2(l) + N_S\Vol(E_{c}).$$
As the above identity holds for $S$ an ideal triangle $T$ we have
$$\Vol (T_1(T)) = 2\pi^2 = 3 \Vol(E_c).$$
Therefore
$$\Vol(T_1(S)) = \sum_{l \in L_S} F_2(l) + \frac{2\pi^2}{3}N_S.$$
The proof of the identity then follows by showing that $F_2(l) = 4.{\mathcal R}(\sech^2(\frac{l}{2}))$ where $\mathcal R$ is the Rogers dilogarithm (in \cite{B11} the index was over all unoriented orthogeodesics so the constant there was $8$). In \cite{B11} this is done by directly computing the integral in the formula \ref{bk_int_form} for $F_2$. We now describe an alternate approach that avoids this computation.

\subsection{Finite identities}
The surface identity is only a finite identity when $S$ is an ideal n-gon. In this case we have a finite orthospectrum $l_1,\ldots, l_k$. We then have
$$2\pi^2(n-2) = \sum_{i} F_2(l_i) + \frac{2\pi^2n}{3}.$$
We rewrite this as
$$\sum_{i} F_2(l_i) = \frac{4\pi^2(n-3)}{3}.$$
We let $R$ be the function defined such that $F_2(l) = 8.R(\sech^2(\frac{l}{2}))$.  Then $R$ satisfies
$$ \sum_{i} R\left(\sech^2\left(\frac{l_i}{2}\right)\right) = \frac{(n-3)\pi^2}{6}.$$
If the ideal polygon $S$ has cyclically ordered vertices $x_i$, then the $i$th side can be identified with the geodesic with endpoints $x_i, x_{i+1}$. Then an orthogeodesic is the perpendicular between geodesic $x_i, x_{i+1}$ and geodesic $x_j, x_{j+1}$ where $|i-j| \geq 2$ which we label $\alpha_{ij}$ of length $l_{ij}$. We define the cross ratio of four points by
$$[a,b;c,d] = \frac{(a-b)(d-c)}{(a-c)(d-b)}.$$
Then a simple calculation shows that
$$ [x_i, x_{i+1}; x_j, x_{j+1}] = [1,-1,-e^{l_{ij}}, e^{l_{ij}}] =  \frac{(1+1)(e^{l_{ij}}+e^{l_{ij}})}{(1+e^{l_{ij}})(1+e^{l_{ij}})} = \sech^2(l_{ij}/2)$$
 Therefore we can rewrite the identity in the finite case as
$$ \sum_{|i-j| \geq 2} R\left([x_i,x_{i+1};x_j,x_{j+1}]\right) = \frac{(n-3)\pi^2}{6}.$$

{\bf Ideal Quadrilaterals and Euler's identity:} For $n=4$, $S$ is an ideal quadrilateral with two ortholengths $l_1, l_2$ and vertices $x_1,x_2,x_3,x_4$. Thus as $[x_1, x_2;x_3,x_4] = 1 -[x_3,x_4;x_1,x_2]$, for any  $0 < x < 1$ we have
$$R(x) + R(1-x) =\frac{\pi^2}{6}.$$
This identity was proved for the dilogarithm function by Euler (see Lewin's book \cite{Lew91} for details).
Also by symmetry $R(1/2) = \pi^2/12$ and similarly $R(0) = 0 , R(1) = \pi^2/6$.

{\bf Ideal Pentagons and Abel's identity:}  If $S$ is a ideal pentagon then there are $5$ orthogeodesics. We send three of the vertices to $0, 1, \infty$ and  the other two to $u,v$ with $0 < u < v < 1$. Then the cross ratios in terms of $u,v$ are
$$u, \qquad 1-v, \qquad \frac{v-u}{v} , \qquad \frac{v-u}{1-u}, \qquad \frac{u(1-v)}{v(1-u)} $$
Putting into the equation we obtain the following equation.
\begin{eqnarray*}
R\left(u\right)+R\left(1-v\right)+R\left(\frac{v-u}{v}\right)+R\left(\frac{v-u}{1-u}\right)+ R\left(\frac{u(1-v)}{v(1-u)}\right)  = \frac{\pi^{2}}{3}.
 \end{eqnarray*}
Letting $x = u/v, y = v$,  we get
\begin{eqnarray*}
R\left(xy\right)+R\left(1-y\right)+R\left(1-x\right)+R\left(\frac{y(1-x)}{1-xy}\right)+ R\left(\frac{x(1-y)}{1-xy}\right)  =  \frac{\pi^{2}}{3}.
 \end{eqnarray*}
Now by applying Euler's identity for $x,y$, we obtain Abel's pentagon identity
\begin{eqnarray*}
R\left(x\right)+R\left(y\right) =R\left(xy\right)+R\left(\frac{y(1-x)}{1-xy}\right)+R\left(\frac{x(1-y)}{1-xy}\right) .
 \end{eqnarray*}

To show $R = {\mathcal R}$ we use the observation of Calegari  in
\cite{Cal10a}, that by a result of Dupont (see \cite{Du87}), the Rogers dilogarithm is
characterized by satisfying the Euler and Abel identities, and therefore  $R = {\mathcal R}$.

\subsection{Proof of the Luo-Tan identity}
We let $S$ be a closed hyperbolic surface. We note that as $\partial
S=\emptyset$, the Basmajian and Bridgeman identities do not make sense as there are no orthogeodesics. Similarly, it is not clear how to extend  the McShane identity to this case as there is no starting point, i.e.,
no boundary component or horocycle to decompose. Furthermore, since the generalization of the McShane identity to cone surfaces has a restriction that all cone angles are $\le \pi$, we cannot deform a cone singularity to a smooth point to obtain an identity. 

However, one can combine two key ideas from the Bridgeman-Kahn and the McShane proofs to obtain an identity, which is what Luo and the second author did. The key idea
from  the Bridgeman-Kahn identity  is to start from any point and
any direction, that is, to define $X = T_1(S)$ with $\mu$ the volume
measure. The key idea we use from the proof of the McShane identity is that instead of flowing in both directions indefinitely, as in the Bridgeman identity, we flow
until we get intersection points. In this way, for a generic vector $v \in T_1(S)$, we construct a geodesic graph $G_v$ with Euler characteristic $-1$  and use this to obtain a decomposition of $X$. From this decomposition we calculate the measures of the components to obtain the Luo-Tan identity.

More precisely, given $v \in T_1(S)$, consider
the unit speed geodesic rays $\gamma_v^+(t)$ and $\gamma_v^-(t)$ $(
t \geq 0)$ determined by exponentiating $\pm v$. If the vector $v$
is generic, then both rays will self intersect transversely by the
ergodicity of the geodesic flow, otherwise, we have $v \in Z \subset X$
where $\mu(Z)=0$.

\begin{figure}[htbp] %  figure placement: here, top, bottom, or page
   \centering
   \includegraphics[width=4in]{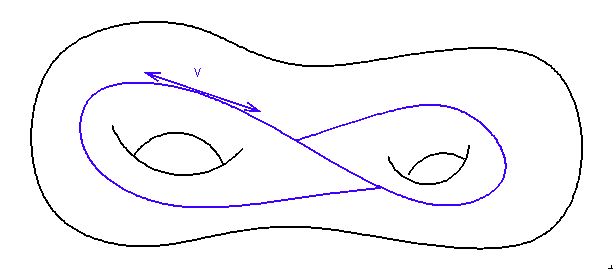} 
   \caption{Graph $G_v$}
   \label{lt1}
\end{figure}

Each $v\in X \setminus Z$ will determine a
canonical graph $G_v$ as follows (see figure \ref{lt1}). Consider the path
$A_t=\gamma_v^-([0,t]) \cup \gamma_v^+([0,t])$ for $t
>0$ obtained by letting the geodesic rays $g_v^{-}$ and
$g_v^+$ grow at equal speed from time $0$ to $t$. Let $t_+>0$ be the
smallest positive number so that $A_{t_+}$ is not a simple arc,
without loss of generality, we may assume that $\gamma_v^+(t_+) \neq
\gamma_v^-(t_+)$ by ignoring a set of measure zero (i.e. putting it
into $Z$). Say $\gamma_v^+(t_+) \in \gamma_v^-([0,t_+]) \cup
\gamma_v^+([0, t_+))$. Next, let $t_-
> t_+$ be the next smallest time so that $\gamma_v^-(t_-) \in \gamma_v^-([0,
t_-)) \cup \gamma_v^+[0, t_+])$. 

\medskip

\begin{definition} The union $\gamma_v^-([0, t_-])
\cup \gamma_v^+([0, t_+])$ is the graph $G_v$ associated
to $v$. 
\end{definition}
From the definition, $G_v$ has  Euler characteristic is $-1$. 
 We call an embedded pair of pants (three hole sphere) or one-hole torus with geodesic boundary in $S$ a {\it simple geometric embedded subsurface}. The following result allows us to decompose $X\setminus Z$ into subsets indexed by the set of simple geometric embedded subsurfaces.

\begin{figure}[htbp] %  figure placement: here, top, bottom, or page
   \centering
   \includegraphics[width=4in]{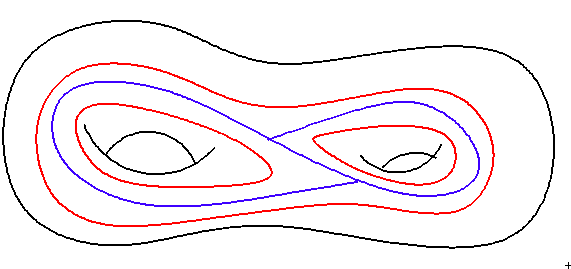} 
   \caption{Subsurface $\Sigma_v$, with $G_v \subset \Sigma_v$}
   \label{lt2}
\end{figure}
\begin{proposition}\label{prop:loopisembedded}(Proposition 3.1 of \cite{LT11})
The graph $G_v$ is contained in a unique simple geometric
embedded subsurface $\Sigma_v\subset S$.

\end{proposition}

\begin{proof}

Cutting the surface $S$ open along $G_v$, we obtain a  surface
whose metric completion $\hat S$ is a compact hyperbolic surface
with convex boundary. The boundary of $\hat S$ consists of 
simple piecewise geodesic loops (corresponding to $G_v$), each boundary has at least one corner. If $\hat \gamma$ is
a  simple piecewise geodesic loop in $\partial \hat S$, it is freely
homotopic to a simple closed geodesic $\gamma$ in $\hat S$ which is
a component of the boundary of   $core( \hat S)$, the convex core of
$ \hat S$. Furthermore $\hat \gamma$ and $\gamma$ are disjoint by
convexity. Therefore, $\hat \gamma$ and $\gamma$ bound a convex
annulus exterior to $core(\hat S)$ and $G_v$ is disjoint from
$core(\hat S)$. The simple geometric subsurface $\Sigma_v$
containing $G_v$ is the union of these convex annuli bounded by
$\hat \gamma$ and $\gamma$ (see figure \ref{lt2}). The Euler characteristic of $\Sigma_v$
is $-1$ by the construction. Furthermore, the surface $\Sigma_v$ is
 unique. Indeed, if
 $\Sigma' \neq \Sigma \subset S$ is a simple geometric subsurface
 so that $G_v\subset \Sigma'$, then $\Sigma'$ has a
 boundary component say $\beta$ which
 intersects one of the boundaries $ \gamma$ of $\Sigma$ transversely (we use here in an essential way the fact that the surfaces $\Sigma$ and $\Sigma'$ are simple).
 Therefore, $\beta$ must intersect the other boundary
 $\hat \gamma$ of the convex annulus described earlier, otherwise we have a hyperbolic bigon, a contradiction. Hence
 it intersects $G_v$ which contradicts $G_v\subset \Sigma'$.
  \end{proof}

 The above discussion gives a
decomposition of  $T_1(S)\setminus Z$ indexed by these simple subsurfaces, namely, for any simple subsurface $\Sigma \subset S$, define $$X_{\Sigma}:=\{v \in X ~:~ G_v \subset \Sigma \}.$$
Then $\mu(X)=\mu(T_1(S))=\sum_P \mu(X_P)+\sum_T\mu(X_T)$ where the first sum is over all simple geometric embedded pairs of pants and the second sum over all simple geometric embedded one-hole tori. It remains to calculate for a simple hyperbolic surface $\Sigma=P$ or $T$ the volume of the set of all unit tangent vectors $v\in T_1(\Sigma)$
so that $G_v$ is strictly contained in $\Sigma$, that is, $G_v$ is a spine for $\Sigma$, this will give us the functions $f$ and $g$ in Theorem $D$.

\subsection{Computing the measures of $X_P$ and $X_T$}
By definition, we have $f(P)=\mu(X_P)$ and $g(T)=\mu(X_T)$. It is complicated to compute the measures directly, it turns out that it is easier to compute $\mu(X_P)$ and $\mu(X_T)$   by calculating the measure of the complementary set in $T_1(\Sigma)$  instead. The idea is that the  vectors $v \in T_1(P)$ in the complementary set can be divided into a small number of disjoint types which can be described quite easily geometrically, hence the complementary set decomposes  into a finite number of disjoint subsets whose measures can be computed in a similar way to the computation for the Bridgeman identity. 

\medskip

Suppose $P$ is a hyperbolic 3-holed sphere, and $i \in\{1,2,3\}$ taken mod $ 3$. Denote the boundary geodesics of $P$ by $L_i$, the simple orthogeodesics from $L_{i+1}$ to $L_{i+2}$ by $M_i$ and the simple orthogeodesic from $L_i$ to itself by $N_i$, see figure \ref{fig6}(a).
 Denote the lengths of the boundary geodesics and the orthogeodesics by the corresponding lower case letters, that is $l_i$, $m_i$ and $n_i$ respectively. 

\begin{figure}[htbp] %  figure placement: here, top, bottom, or page
   \centering
   \includegraphics[width=2.25in]{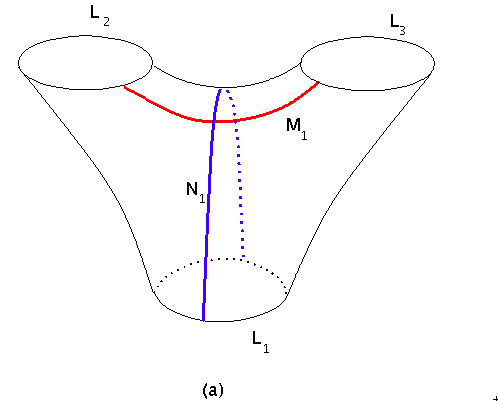}   \includegraphics[width=1.5in]{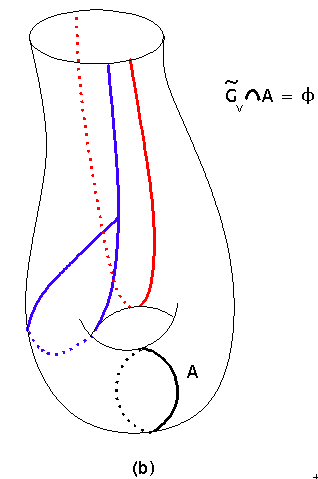} 
   \caption{3-holed spheres and 1-holed tori}
   \label{fig6}
\end{figure}

For $v \in T_1(P)$ we define $\gamma_v$  to be the maximal geodesic arc in $P$ tangent to $v$. We modify the definition of $G_v$ appropriately, to take into account the fact that $\partial P \neq \emptyset$,  that is, $$\hat{G}_v=\gamma_v^-([0, t_-])
\cup \gamma_v^+([0, t_+])$$ 
as before but the times $t_+$ and $t_-$ may also occur when the geodesic hits $\partial P$ in which case we stop generating the geodesic in that direction.

From the definitions, it is clear that $\hat{G}_v \subset \gamma_v$, and $v \in X_P$ if and only if $\hat{G}_v \cap \partial P=\emptyset$. In particular, if $v \in T_1(P) \setminus X_P$, and $v \not\in Z$, then $\hat{G_v}$ is either simple, with both endpoints on $\partial P$, or has one endpoint on $L_i$ and the other end is a loop freely homotopic to $L_j$, $j \neq i$, we call this a {\em lasso} around $L_j$ based at $L_i$ (note that the loop cannot be homotopic to $L_i$). The following gives a decomposition of  $T_1(P) \setminus X_P$ into finitely many types:

\begin{itemize}
\item Define $H(M_i)=\{v\in T_1(P) ~|~ \gamma_v \sim M_1~ {\hbox{rel. boundary}}\}$.  If $v \in H(M_i)$, then $\gamma_v$ is simple,  $\hat{G}_v=\gamma_v$, and $v \not\in X_P$. The measure of these sets,  computed by Bridgeman in \cite{B11}, depend only on $m_i$ and is given by $\mu(H(M_1))=8.{\mathcal R}\left(\sech^2\frac{m_i}{2}\right)$.

\item Define $H(N_i)=\{v\in T_1(P) ~|~ \gamma_v \sim N_1~ {\hbox{rel. boundary}}\}$.  If $v \in H(N_i)$,  then $\gamma_v$ intersects $M_i$   exactly once, the point of intersection divides $\gamma_v$ into two components $\gamma_v^+$ and $\gamma_v^-$. This case is  more complicated as $\gamma_v$ may have arbitrarily many self-intersections, however, $\gamma_v^+$ and $\gamma_v^-$ are both simple. This can be seen by cutting $P$ along $M_i$ to obtain a convex cylinder bounded on one side by $L_i$ and the other by a piecewise geodesic boundary. Then both $\gamma_v^+$ and $\gamma_v^-$ are geodesic arcs from one boundary of the cylinder to the other, so must be simple. In particular, in the construction of $\hat{G}_v$, we see that in this case, $\hat{G}_v$ is either a simple geodesic arc from $L_i$ to itself, or it must be a lasso based at $L_i$. That is, for all such $v$, $\hat{G}_v$ intersects $L_i$ so $v \not\in X_P$. Again, by Bridgeman the measure is $\mu(H(N_i))=8.{\mathcal R}\left(\sech^2\frac{n_i}{2}\right)$.
Note that these sets are disjoint from those in the first case.

\item The remaining case is when $\hat{G}_v$ is a lasso, but does not come from case two above, we call these {\it true} lassos. For $\{i,j,k\}=\{1,2,3\}$ distinct, let $$W(L_i,M_j)=\{v\in T_1(P) ~|~ \hat{G}_v ~\hbox{is a true lasso around $L_i$ based at $L_k$ }\}.$$ Consider say that $\hat{G}_v$ is a lasso based at the point $q$ on  $L_1$ with (positive) loop around $L_2$, but such that $\gamma_v$ is not homotopic relative to the boundary to $N_1$.  It is convenient to understand the set of $v$ generating such $\hat{G}_v$  in the universal cover, $\tilde P \subset \Hyp^2$ (we use the upper half space model). We work with the following setting:\\
Consider a fundamental domain for $P$ in $\Hyp^2$ consisting of two adjacent right angled hexagons $H$ and $H'$ such that $H$ is bounded by $\tilde L_1$, $\tilde L_2$ and $\tilde L_3$ and $H'$ is bounded by $\tilde L_2$, $ \tilde L_1'$ and $\tilde L_3$, see figure \ref{fig7}.  Given $x, y \in \Hyp^2\cup \partial \Hyp^2$, $x \neq y$, let $G[x,y]$ denote the  geodesic in $\Hyp^2$ from $x$ to $y$. 
Normalize so that  $\tilde L_1=G[c,d]$, $\tilde L_2=G[\infty, 0]$, and $\tilde L_3=G[e,f]$, where $e,f, c, d \in \mathbb R$ satisfy $0<e<f<c<d$. Further normalize  so that $ \tilde L_1'=G[1,a]$  where $1<a<e$.   By elementary calculations, we have  $c=e^{l_2},~d=e^{l_2} \coth^2(m_3/2)$.

\begin{figure}[htbp] %  figure placement: here, top, bottom, or page
   \centering
   \includegraphics[width=4in]{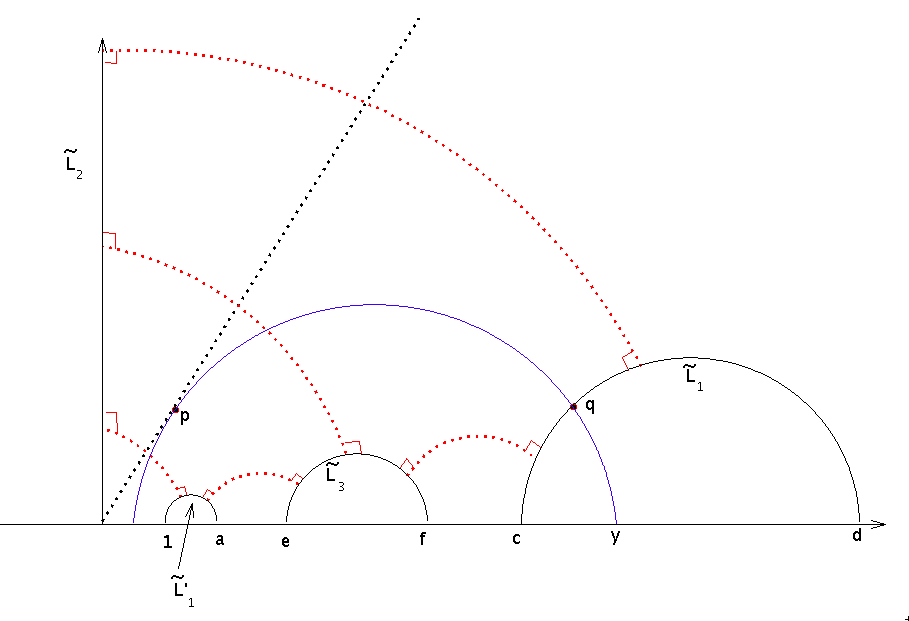} 
   \caption{Universal Cover of $P$ in $\Hyp^2$}
   \label{fig7}
\end{figure}
We can choose a lift of $\hat{G}_v$ so that  the end point $q$ lies on $\tilde L_1=G[c,d]$. Let $G[y,x]$, $y, x \in \mathbb{R}^2$ be the complete geodesic in $\Hyp^2$ containing this lift, in particular, by construction $y \in (c,d)$.  We claim that there is a unique lift so that $0<x<1$. The cases $x=0$ are $x=1$ are important limiting cases. When $x=0$, $\hat{G}_v$ is a simple geodesic of infinite length which spirals around $L_2$ in the positive direction, when $x<0$, $\hat{G}_v$ is homotopic to $M_3$, so is a simple finite geodesic arc from $L_1$ to $L_2$. When $x$ is small and positive, $\hat{G}_v$ is a lasso with positive loop around $L_2$. When $x=1$, $\gamma_v$ spirals around $L_1$ with infinite number of self intersections, and $\hat{G}_v$ is still a lasso from $L_1$ with positive loop around $L_2$. As $x$ ranges between $1$ and $a$, $\gamma_v$ is homotopic to $N_1$ and $\hat{G}_v$ ranges from being a lasso with positive loop around $L_2$ to being simple and then a lasso with negative loop around $L_3$. In particular,  if $\hat{G}_v$ is a lasso with end point at $L_1$ and positive loop around $L_2$, but $\gamma_v$ is not homotopic to $N_1$ we have a unique lift so $0<x<1$ and $c<y<d$. However, not all tangent vectors on $T_1(G[y,x])$ generate $\hat{G}_v$. By the construction of $\hat{G}_v$, the midpoint of the loop, $p$ of $\hat{G}_v$ is a critical point, all vectors which project to the longer part of $\hat{G}_v$ will generate $\hat{G}_v$, those which project to the shorter part (consisting of half of the loop) will generate a different $\hat{G}_v$, by construction. Also,  the loop of $\hat{G}_v$ bounds with $L_2$ a convex cylinder in $P$ so $p$ is the point on the loop which is closest to $L_2$. Hence, $p$ is the point on $G[y,x]$ such that the ray $0p$ is tangent to $G[y,x]$. To 
calculate the volume, we need to integrate the length of the geodesic from $q$ to $p$ over the set of all geodesics in $\Hyp^2$ with endpoints $x,y$ where $0<x<1$ and $c<y<d$, with respect to the Liouville measure on the space of geodesics. Since both forward and backward vectors on this set generate $\hat{G}_v$, we multiply the result by 2, that is we want twice the volume of the set $\Omega$ where
\begin{equation}\label{eqn:123} \Omega = \{ v \in T_1(\Hyp^2)~ |~ \\ v \in T_1(G[q,p]), \\
0<x<1,\\ c<y<d\}.\end{equation}
The computation of $Vol(\Omega)$ is elementary but somewhat messy, the final integral which needs to be computed is given by the following:

\begin{proposition}(Prop 4.1 of \cite{LT11}) 
The volume of $\Omega$ is given by

\begin{equation}\label{eqn:VolofOmega}
 \int_0^1 \left(\int_c^d  \frac{ \ln \left| \frac{
y(x-c)(x-d)}{x(y-c)(y-d)}\right|}{(y-x)^2} dy\right) dx.
\end{equation}
\end{proposition} 

Note that the above integral  only depends on $c$ and $d$ which are given by
$c=e^{l_2},~d=e^{l_2} \coth^2(m_3/2)$. More generally, define the Lasso function $La(l,m)$ to be the above integral where $c=e^{l},~d=e^{l} \coth^2(m/2)$. We have:

\begin{proposition}\label{lem:simplificationforlassofunction}(Prop 4.6 of \cite{LT11})
The lasso function $La(l,m)$ is given by  $$La(l,m) =2\left({\mathcal R}(y) -{\mathcal R}\left(\frac{1-x}{1-xy}\right) +
{\mathcal R}\left(\frac{1-y}{1-xy}\right)\right)$$ where $x=e^{-l}$ and $y=\tanh^2(m/2)$, ${\mathcal R}(z)$ is the Roger's dilogarithm.
\end{proposition}
As remarked earlier, we need to take twice the volume of $\Omega$ to obtain the volume of the set of vectors generating true lassos from $L_1$ with a positive loop around $L_2$, the volume for those with a negative loop around $L_2$ is the same by symmetry, hence we have:
$$Vol(W(L_i,M_j))=4La(l_i,m_j).$$
Now $\mu(X_P)$ can be computed by subtracting away from $\mu(T_1(P))$ the measures of the sets $H(M_i)$, $H(N_i)$, $i=1,2,3$ and $W(L_i,M_j)$, $i\neq j$. To obtain the expression in Theorem D, we use some of the pentagon relations satisfied by the Roger's dilogarithm function, for details, see \cite{LT11}.

\bigskip

The computation of $\mu(X_T)$ for an embedded one-holed torus is similar, with an extra observation. Again, $v \not\in X_T$ if $\hat{G}_v \cap \partial T \neq \emptyset$, ($\partial T$ has only one component). In this case $\hat{G}_v$ is either a simple geodesic arc from $\partial T$ to itself, or is a lasso based at $\partial T$ with a loop homotopic to an essential, non boundary parallel simple closed geodesic. In either case, $\hat{G}_v$ is disjoint from a {\it unique} simple closed geodesic $A \subset T$, see figure 6b. Cutting along $A$ produces a pair of pants $P_A$, and from the previous calculations, we can calculate the set of all $v \in T_1(P_A)$ such that $\hat{G}_v$ intersects $\partial T$ but not the other two boundary components of $P_A$. Summing up the measures over all possible simple closed geodesics $A$ we obtain the measure of the complement of $X_T$ in $T_1(T)$. Again, by manipulating the expressions using the identities for $\mathcal R$, we obtain the expression in Theorem D for $g(T)$, see \cite{LT11} for details. \qed

\end{itemize}

\section{Moments of hitting function}
We consider all four identities and their associated measure space $(X,\mu)$. Associated to this we have a hitting function $L:X \rightarrow \Real_+$ where $L(x)$ is the length of the geodesic arc associated to $x$.  The function $L$ is measurable and we can  consider it as a random variable with respect to the measure $\mu$. For $k \in \Z$, the $k^{th}$ moment of $L$ with respect to the measure $\mu$ is then
$$M_k(X) = (L_*\mu)(x^k)  = \int_X L^k(x) d\mu.$$
In particular $M_0(X) = \mu(X)$. Also the measurable decomposition $X = Z \cup \bigcup_i X_i$ gives us a formula
$$M_k(X) = \sum_i \int_{X_i} L^k(v) d\mu.$$ 
In each identity it is easy to again show that each integral in the summation on the right only depends on the spectrum associated with the identity. Therefore one can find smooth functions $F_k$ such that 
$$M_k(X) = \sum_{l \in S} F_k(l)$$ 
where $S$ denotes the spectrum of the given identity. In particular, this formula is the original identity in the case $k = 0$.
Also the average length of the geodesic associated with an element of $X$, called the {\em average hitting time} $A(X)$, is then given by
$$A(X) = \frac{M_1(X)}{\mu(X)} = \frac{M_1(X)}{M_0(X)}.$$

\subsection{Moments of Bridgeman-Kahn Identity}
In a recent paper, we consider the moments of the Bridgeman-Kahn identity and show that both the Bridgeman-Kahn and Basmajian identities arise as identities for its moments.The Bridgeman-Kahn identity  obviously arises as the identity for the $k=0$ moment. We  show that the Basmajian identity appears as the identity for the $k=-1$ moment,  giving a link between the two identities. 
 We also derive an integral formula for the moments and an explicit formula for $A(X)$ in the surface case.

\medskip
\begin{theorem}{(Bridgeman-Tan, \cite{BT13})}
 There exists smooth functions $F_{n,k}: \Real_+ \rightarrow \Real_+$ and constants $C_n > 0$  such that if $X$ is a compact hyperbolic n-manifold with totally geodesic boundary $\partial X \neq \emptyset$, then
\begin{enumerate}
\item The moment $M_k(X)$ satisfies
$$M_k(X) = \sum_{l \in L_X} F_{n,k}(l)$$
\item $M_{0}(X) = \Vol(T_1(X))$ and  the identity is the Bridgeman-Kahn identity.
\item $M_{-1}(X) = C_n.\Vol(\partial X)$ and  the identity for $M_{-1}(X)$ is the   Basmajian identity.
\item The average hitting time $A(X)$ satisfies 
$$A(X) = \frac{1}{\Vol(T_1(X))}\sum_{l \in L_X} F_{n,1}(l) = \sum_{l \in L_X} G_{n}(l).$$
\end{enumerate}
\label{moments}
\end{theorem}

In the surface case we obtain an explicit formula for the function $G_2$ and hence $A(X)$ in terms of polylogarithms. Furthermore, besides compact surfaces obtained as quotients of Fuchsian groups, the identity holds more generally for  finite area surfaces with boundary cusps.

\begin{theorem}{(Bridgeman-Tan, \cite{BT13})}
Let $S$ be a finite area hyperbolic surface with non-empty totally geodesic boundary. Then
$$A(S) = \frac{1}{2\pi\mbox{Area}(S)}\left(\sum_{l \in L_S} F\left(\sech^2\frac{l}{2}\right) + 6\zeta(3)C_S\right)$$
where
\begin{eqnarray*}F(a) &=&  -12\zeta(3)-\frac{4\pi^2}{3}\log(1-a)+6\log^2(1-a)\log(a)-4\log(1-a)\log^2(a)\\
&& \qquad -8\log\left(\frac{a^2}{1-a}\right)Li_2(a)
+24Li_3(a)+12Li_3(1-a),
\end{eqnarray*}
for $Li_k(x)$ the $k^{th}-$polylogarithm function, and $\zeta$ the Riemann $\zeta-$function.
\label{avehit}
\end{theorem}

\subsection{Moments of Basmajian Identity}
In the recent preprint \cite{V13}, Vlamis  considers the moments for the Basmajian identity.

\begin{theorem}{(Vlamis, \cite{V13})}
Let $X$ be a compact hyperbolic manifold with totally geodesic boundary and $m_k(X)$ be the moments of the boundary hitting function with respect to lebesgue measure on $\partial M$. Then
$$m_k(X) =  \sum_{l \in L_X} f_{n,k}(l)$$
where
$$f_{n,k}(l) = \Omega_{n-2} \int_0^{\log \coth(l/2)} \left(\log\left(\frac{\coth l + \cosh r}{\coth l - \cosh r}\right)\right)^k\sinh^{n-2}r dr$$
and $\Omega_{n-2}$ is the volume of the unit  $(n-2)$-dimensional sphere.
\end{theorem}

Vlamis derives an explicit formula in odd-dimensions and further derives a formula for the first moment in the case of a surface $S$.

\begin{theorem}{(Vlamis, \cite{V13})}
Let $S$ be a compact hyperbolic surface with non-empty totally geodesic boundary. Then
$$A(S) = \frac{1}{2\pi\mbox{Area}(S)}  \sum_{l \in L_S} \left( Li_2\left(-\tanh^2\frac{l}{2}\right) -Li_2\left(\tanh^2\frac{l}{2}\right) + \pi^2/4\right).$$
\end{theorem}

\section{The Bowditch proof of the McShane identity and generalizations}
Bowditch gave an algebraic-combinatorial proof of the original McShane identity for the punctured torus in \cite{Bow96}, and extending the method, proved variations for punctured torus bundles in \cite{Bow97} and representations of the punctured torus group (including the quasi-fuchsian representations) satisfying some conditions he calls the Q-conditions   in \cite{Bow98}. One advantage of the proof is that it avoids the use of the Birman-Series result on the Hausdorff dimension of the set of points on simple geodesics on a hyperbolic surface. Akiyoshi, Miyachi and Sakuma refined the identity for punctured torus groups in \cite{AMS2004} and found variations for quasi-fuchsian punctured surface groups in \cite{AMS2006}.  

Let $T$ be a once punctured hyperbolic torus, and $\pi:=\pi(T)=\langle X,Y\rangle$ the fundamental group of $T$, a free group on two generators, and $\rho: \pi \rightarrow {\mathrm{PSL}(2,{\mathbb R})}$ the holonomy representation. Define an equivalence relation $\sim$  on $\pi$, by $X \sim Y$ if $X$ is conjugate to $Y$ or $Y^{-1}$. The classes correspond to free homotopy classes of closed curves on $T$, let $\Omega \subset \pi/\sim$  be the set of classes corresponding to {\em essential, simple, non-peripheral} simple closed curves on $T$. Classes in $\Omega$ have representatives which form part of a generating pair for $\pi$, we call them primitive classes. For $X \in \Omega$, $x:={\rm tr}\rho(X)$ is well defined up to sign, and is related to the length $l$ of the unique geodesic representing the class by
$$\cosh^2( l/2)=x^2/4.$$
 McShane's orginal identity for the once-punctured torus then has the form
\begin{equation}\label{eqn:Bowditch}
\sum_{X \in \Omega}h(x)=\sum_{X \in \Omega}\left(1-\sqrt{1-\frac{4}{x^2}}\right)=\frac{1}{2}
\end{equation}
where we let $x={\rm tr} \rho(X)$, and $h(x)=1-\sqrt{1-\frac{4}{x^2}}$.
This is the form which was proven by Bowditch, and generalized to type-preserving representations (i.e., ${\rm tr} \rho(XYX^{-1}Y^{-1})=-2$) of $\pi$ into ${\mathrm{SL}(2,{\mathbb C})}$ satisfying the following conditions which he calls the Q- conditions, we call here the BQ-conditions:

\begin{definition}A represention $\rho$ from $\pi$ into  ${\mathrm{SL}(2,{\mathbb C})}$ satisfies the Bowditch Q-conditions (BQ-conditions) if
\begin{enumerate}
\item ${\rm tr} \rho (X) \not \in [-2,2]$ for all $X \in \Omega$;
\item  $|{\rm tr} \rho (X)| \le 2$ for only finitely many (possibly none) $X \in \Omega$.
\end{enumerate}
\end{definition}
The basic idea of the proof was to represent the values taken by elements of $\Omega$ in an infinite trivalent tree. This arises from the fact that $\Omega$ can be identified with ${\mathbb Q} \cup \infty$ by considering the slopes of the curves in $T$, and the action of the mapping class group of $T$ on this set is essentially captured by the Farey tessellation, whose dual is an embedded infinite trivalent tree in $\Hyp$. In this way, $\Omega$ is identified with the set of complementary regions of the dual tree, and the values ${\rm tr}\rho(X)$ for $X \in \Omega$  satisfy  vertex and edge relations which come from the Fricke trace identities. Using this function and the vertex and edge relations, Bowditch was able to cleverly assign a function on the directed edges of the tree which satisfied some simple conditions and applied a Stoke's theorem type argument to prove the identity.

 Subsequently, Tan, Wong and Zhang extended the Bowditch method to prove versions of the identity for general representations of the free group $\pi$ into ${\mathrm{PSL}(2,{\mathbb C})}$ satisfying the Q-conditions and closed hyperbolic three manifolds obtained from hyperbolic Dehn surgery in \cite{TWZ08, TWZ06b}. 
%These identities generalize the Mirzakhani \cite{Mir07} and Tan, Wong Zhang \cite{TWZ06, TWZ08b} variations of the identity for hyperbolic surfaces with geodesic boundaries and/or cone points, and Schottky representations. 
Specifically, let $\rho \in {\rm Hom }(\pi, {\mathrm{SL}(2,{\mathbb C})})$ be a representation satisfying the BQ-conditions and let $\tau={\rm tr} \rho(XYX^{-1}Y^{-1})$. We call $\rho$ a $\tau$-representation. Define the function ${\mathfrak h}:={\mathfrak h}_{\tau}$ as follows:

For $\tau \in {\mathbf C}$, set
$\nu=\cosh^{-1}(-\tau/2)$. We define 

\centerline{${\mathfrak h}={\mathfrak h}_\tau: {\mathbf C}
\backslash \{\pm \sqrt{\tau+2}\} \rightarrow {\mathbf C}$}

\noindent by
\begin{eqnarray}
{\mathfrak h}(x)
&=&2 \tanh^{-1}\bigg(\frac{\sinh\nu}{\cosh\nu+e^{l(x)}}\bigg)
\label{eqn:frak h(x)=2tanh^-1} \\
&=&\log\frac{e^{\nu}+e^{l(x)}}{e^{-\nu}+e^{l(x)}} \label{eqn:frak h(x)=log I} \\
&=&\log\frac{1+(e^{\nu}-1)\,h(x)}{1+(e^{-\nu}-1)\,h(x)},
\label{eqn:frak h(x)=log II}
%&=&\log\bigg(\frac{e^{\nu}+h(x)^{-1}-1}{e^{-\nu}+h(x)^{-1}-1}\bigg)\\
%&=&\log \bigg(\frac{e^{\nu}+e^{2\cosh^{-1}(x/2)}}{e^{-\nu}+e^{2\cosh^{-1}(x/2)}}\bigg)\\
%&=&\log \bigg(\frac{(-\tau)^{-1}h(-\tau)^{-1}+h(x)^{-1}-1}{(-\tau)h(-\tau)+h(x)^{-1}-1}\bigg).
\end{eqnarray}
where (\ref{eqn:frak h(x)=2tanh^-1}), (\ref{eqn:frak h(x)=log I}) and  (\ref{eqn:frak h(x)=log II}) are equivalent (see \cite{TWZ08} for details). 
We have:

\begin{theorem}\label{thm:TWZ}(Tan-Wong-Zhang, Theorem 2.2 , \cite{TWZ08})
Let $\rho: \pi \rightarrow {\rm SL}(2, \mathbb C)$ be a
$\tau$-representation {\rm(}where $\tau \neq 2${\rm)} satisfying
the BQ-conditions. Set $\nu=\cosh^{-1}(-\tau/2)$. Then
\begin{eqnarray}\label{eqn:TWZ}
\sum_{X \in  \Omega}{\mathfrak h} (x)=\nu  \mod 2 \pi i,
\end{eqnarray}
where the sum converges absolutely, and $x:={\rm tr}\rho(X)$.
\end{theorem}
When $\rho$ arises as the holonomy of a hyperbolic structure on a one-hole torus with geodesic boundary, or with a cone point, then the above is equivalent to the Mirzakhani \cite{Mir07} and the Tan-Wong-Zhang \cite{TWZ06} variations of the McShane identity.

\medskip

Recently, revisiting the original Bowditch proof, Hu, Tan and Zhang proved new variations of the identity for  representations of  $\pi$ into ${\mathrm{PSL}(2,{\mathbb C})}$ satisfying the BQ-conditions \cite{HTZ13}, we have:

\begin{theorem}\label{thm:HTZ1}(Hu-Tan-Zhang, \cite{HTZ13})
Let $\rho: \pi \rightarrow {\rm SL}(2, \mathbb C)$ be a
$\tau$-representation {\rm(}where $\tau \neq 2${\rm)} satisfying
the BQ-conditions. Let $\mu=\tau+2$. Then
\begin{eqnarray}\label{eqn:TWZ}
\sum_{X \in  \Omega}{g} (x)=\sum_{X \in  \Omega}\left(1-\frac{3x^2-2\mu}{3(x^2-\mu)}\sqrt{1-\frac{4}{x^2}}\right)=\frac{1}{2},
\end{eqnarray}
where the sum converges absolutely, and $x:={\rm tr}\rho(X)$.
\end{theorem}

Note that the type-preserving case occurs when $\tau=-2$ so $\mu=0$ and the above reduces to the original McShane identity in this case.

They also extended this in \cite{HTZ13b} to identities for orbits of points in ${\mathbb C}^n$ under the action of the Coxeter group $G_n$ generated by $n$-involutions which preserve the varieties defined by the Hurwitz equation.
 At the moment, it is not clear what is the underlying geometric interpretation of these identities.

%We give a brief description of the method in this section.
%
%The basic combinatorial set-up here can be described as follows: 
%
%We start with a infinite trivalent tree ${\mathcal T}$  properly embedded in the plane, associated to this is the set $V$ of vertices, the set $E$ of edges (respectively $\vec E$ of directed edges) and the set $\Omega$ of complementary regions. There is a natural tri-coloring of $E$ and $\Omega$ (unique up to permutation) by the set $\{1,2,3\}$ such that every vertex $v \in V$ is adjacent to three edges with different colors, and if $X\in \Omega$ and $e \in E$ is adjacent to $X$ then $X$ and $e$ have different colors.

\section{Concluding remarks}
We have shown that the identities obtained by Basmajian, McShane, Bridgeman-Kahn and Luo-Tan are obtained by considering decompositions of certain sets $X$ with finite measure $\mu$ associated to the manifold $M$ obtained by considering some kind of geodesic flow, either from the boundary, or from the interior. Typically, there is a subset of measure zero which is complicated but which does not contribute to the identity, and the subsets of non-zero measure in the decomposition are indexed by some simple geometric objects on the manifold. The measures of each subset typically depends only on the local geometry and data, and not on the global geometry of $M$, which may be easy or fairly complicated to compute. One can apply this general philosophy to try to obtain other interesting identities.
There are also many other interesting directions for further research and exploration in this area, we list a number of them below:
\begin{enumerate}
\item Find good applications for the identities, for example use them to say something about the moduli space of hyperbolic surfaces. The McShane-Mirzakhani identity was an important ingredient in the work of Mirzakhani in the study of the Weil-Petersson geometry of the moduli space of bordered Riemann surfaces. It would be interesting to find similar applications for the Basmajian, Bridgeman and Luo-Tan identities.

\item Generalize the McShane/Luo-Tan identities to hyperbolic manifolds of higher dimension or to translation surfaces.

\item The Basmajian and Bridgeman-Kahn identities do not generalize easily to complete finite volume hyperbolic surfaces with cusps (or cone points) as in this case the limit set does not have measure zero. It would be interesting to find other interesting decompositions of horocycles which would generalize these identities, some progress on this has been made recently by Basmajian and Parlier \cite{BasPar}.

\item It would be interesting to extend the Bowditch method to general surfaces of genus $g$ with $n$ cusps, some progress has been been recently by  Labourie and the second author, \cite{LaT13}.

\item Give a geometric interpretation of the variation of the McShane identity (Theorem \ref{thm:HTZ1}) obtained by  Hu-Tan-Zhang using the extension of the Bowditch method, and also the identity for the  $n$-variables case.
\item Extend the identities to other interesting representation spaces, see for example the work of Labourie and McShane in \cite{LaMcS09}, and Kim, Kim and Tan in \cite{KKT12}
\item Extend the moment generating point of view, and derive formulae for the moments of the McShane and Luo-Tan identities.
\item Obtain identities for general closed hyperbolic manifolds.
\item Analyze the asymptotics of the functions $f$ and $g$ in the Luo-Tan identity and use it to study the asymptotics of embedded simple surfaces in a hyperbolic surface.
\end{enumerate}

\end{document}